\documentclass[twoside]{article}

\usepackage{amsmath,amsthm,amssymb}

\usepackage{newtxtext,newtxmath}

\usepackage[text={12.5cm,19cm},centering,paperwidth=17cm,paperheight=24cm]{geometry}

\usepackage[fontsize=10.4pt]{scrextend}
\usepackage[mathscr]{eucal}

\renewcommand{\a}{\alpha}
\newcommand{\be}{\beta}

\newcommand{\G}{\Gamma}
\newcommand{\de}{\delta}
\newcommand{\D}{\Delta}
\newcommand{\e}{\epsilon}
\newcommand{\ve}{\varepsilon}

\newcommand{\y}{\eta}

\newcommand{\io}{\iota}
\newcommand{\ka}{\kappa}

\newcommand{\la}{\lambda}
\newcommand{\La}{\Lambda}
\newcommand{\m}{\mu}
\newcommand{\n}{\nu}
\newcommand{\x}{\xi}
\newcommand{\X}{\Xi}

\newcommand{\s}{\sigma}
\newcommand{\Si}{\Sigma}
\newcommand{\vs}{\varsigma}

\newcommand{\vf}{\varphi}
\newcommand{\F}{\Phi}
\newcommand{\hi}{\chi}

\newcommand{\om}{\omega}
\newcommand{\Om}{\Omega}
\newcommand{\U}{\Upsilon}

\newcommand{\R}{{\mathbb R}}

\newcommand{\N}{{\mathbb{N}}}

\newcommand{\db}{{\mathbf d}}
\newcommand{\eb}{{\mathbf e}}

\newcommand{\kb}{{\mathbf k}}

\newcommand{\nb}{{\mathbf n}}

\newcommand{\tb}{{\mathbf t}}

\newcommand{\xb}{{\mathbf x}}
\newcommand{\yb}{{\mathbf y}}

\newcommand{\Db}{{\mathbf D}}

\newcommand{\Fb}{{\mathbf F}}

\newcommand{\Jb}{{\mathbf J}}

\newcommand{\Lb}{{\mathbf L}}
\newcommand{\Mb}{{\mathbf M}}
\newcommand{\Nb}{{\mathbf N}}

\newcommand{\Qb}{{\mathbf Q}}

\newcommand{\Tb}{{\mathbf T}}

\newcommand{\Zb}{{\mathbf Z}}

\newcommand{\AF}{\mathfrak A}

\newcommand{\BFB}{\mathfrak B}

\newcommand{\dF}{\mathfrak d}

\newcommand{\WF}{\mathfrak W}
\newcommand{\XF}{\mathfrak X}

\newcommand{\Ac}{{\mathcal A}}

\newcommand{\Dc}{{\mathcal D}}
\newcommand{\Ec}{{\mathcal E}}

\newcommand{\Hc}{{\mathcal H}}

\newcommand{\Kc}{{\mathcal K}}
\newcommand{\Lc}{{\mathcal L}}
\newcommand{\Mc}{{\mathcal M}}

\newcommand{\Yc}{{\mathcal Y}}
\newcommand{\Zc}{{\mathcal Z}}

\newcommand{\dist}{{\rm dist}\,}

\newcommand{\supp}{\hbox{{\rm supp}}\,}

\newcommand{\codim}{\operatorname{codim\,}}

\newcommand{\sign}{\operatorname{sign\,}}

\def\titlerunning#1{\gdef\titrun{#1}}

\makeatletter
\def\author#1{\gdef\autrun{\def\and{\unskip, }#1}\gdef\@author{#1}}
\def\address#1{{\def\and{\\\hspace*{15.6pt}}\renewcommand{\thefootnote}{}\footnote{#1}}\markboth{\autrun}{\titrun}}
\makeatother

\def\email#1{email: \href{mailto:#1}{#1} }

\newenvironment{dedication}{\itshape\center}{\par\medskip}
\newenvironment{acknowledgments}{\bigskip\small\noindent\textit{Acknowledgments.}}{\par}

\newtheorem{thm}{Theorem}[section]
\newtheorem{cor}[thm]{Corollary}
\newtheorem{lem}[thm]{Lemma}

\theoremstyle{definition}
\newtheorem{defin}[thm]{Definition}

\numberwithin{equation}{section}

\usepackage[hyperfootnotes=false,colorlinks=true,allcolors=blue]{hyperref}

\begin{document}

\titlerunning{Eigenvalues of singular measures}

\title{\textbf{Eigenvalue estimates and asymptotics for weighted pseudodifferential operators with singular measures in the critical case}}

\author{Grigori Rozenblum  \and Eugene Shargorodsky}

\date{}
\maketitle

\address{G. Rozenblum: Chalmers Univ. of Technology; EIMI, and St.Petersburg State University; \email{grigori@chalmers.se} \and E.Shargorodsky: King's College London
and Technische Universit\"at Dresden; \email{eugene.shargorodsky@kcl.ac.uk} }

\begin{dedication}
To Ari Laptev, with best wishes, on the occasion of his 70th birthday
\end{dedication}


\section{Introduction}
We study the eigenvalue distribution for self-adjoint compact operators of the type $\AF^* P \AF$, were $\AF$ is a pseudodifferential operator of negative order $-l$ in a domain $\Om\subset \R^\Nb$ and $P$ is a  signed measure in $\Om$. A number of  spectral problems can be reduced to this one, the most important one, probably,  being $\la(-\Delta)^{l} u = Pu$, closely related to the Schr\"odinger operator.  In particular, if $P$ is an absolutely continuous measure, the spectral asymptotics for such operators has been justified under rather mild conditions imposed on $P$, see \cite{BS}. The case of a singular measure on $\Om$ is not that well studied.  In 1951, M.G.Krein discovered that for the `singular string' described by the equation $-\la u''= P u$ with the Dirichlet boundary conditions at the endpoints of an interval and with $P$ being a Borel measure, the leading  term in the  asymptotics of the eigenvalues is determined by the absolutely continuous part $P_{ac}$ of $P$ only, while the singular part $P_{sing}$ makes a weaker contribution. Further on, in papers by M.Sh.~Birman, M.Z. Solomyak, and V.V. Borzov, see, e.g.  \cite{BS}, this property of singular measures was proved for a wide class of `high order' spectral problems, in particular for $\la (-\Delta)^l u=P u$ in a domain  $\Om\subset\R^\Nb$, provided $2l>\Nb$.  For  `low order problems', $2l<\Nb$, the influence of the singular part of the measure is different. The known cases concern $P_{sing}$ concentrated on a surface of codimension 1, and here, in the opposite, it makes a leading order contribution to  the spectral asymptotics, see \cite{Grubb}. The intermediate, `critical' case $2l=\Nb$ has been studied even less (it is the common wisdom that for many questions in spectral theory this case is the hardest one). Until very recently, there were very few results here. In dimension $\Nb=2$, if a measure $P$ is concentrated at the boundary of $\Om$ (which is equivalent to the Steklov problem) or on a smooth curve inside $\Om$, the     eigenvalue asymptotics has the same order as for the regular Dirichlet problem, namely, $\la_k\sim C k^{-1}$, see also general results  in \cite{Agr}, \cite{KozhYak}, \cite{Grubb}.

Recently, the interest to the critical case has revived, due to some new applications, see, e.g., \cite{ShStokes}, \cite{LSZ}. So, a class of singular measures was considered in  \cite{KarSh}. For a singular measure $P=V\m$ in $\R^2$, where $\m$ is Ahlfors $\a$-regular, $\a\in(0,2),$ and $V$ belongs to a certain Orlicz class, an estimate for the eigenvalues of the problem $-\la\Delta u=P u$ was obtained, $|\la_k|\le C(V,\m) k^{-1}$. Thus, unlike other cases, the order in the eigenvalue \emph{estimate} does not depend on the  (Hausdorff) dimension of the support of the measure. The question of sharpness of these estimates was not touched upon. In another field,  the critical case was involved in studies related  to noncommutative integration, see \cite{LSZ}.

In this paper, we consider a class of singular measures in the critical case and prove order sharp estimates and asymptotic formulae for eigenvalues for the corresponding spectral problems. Namely, the measure $P=P_{sing}$ is supposed to be supported on a compact Lipschitz surface $\Si$ of codimension $\dF$ in $\R^\Nb$ and absolutely continuous with respect to the surface measure $\m_\Si$ induced by the embedding of $\Si$ into $\R^{\Nb}$, $P=V\mu_\Si$.  We find that the  eigenvalues $\la_k$ have asymptotics  of order $k^{-1}, $ so the order of asymptotics does not depend on the dimension or codimension of the surface. We consider the case of  compactly supported measures only and therefore do not touch upon  effects related with infinity, which, as it is known, may influence the eigenvalue behavior drastically, even for a nice measure, see, e.g.,  \cite{BLS}. If $P$ is concentrated on several surfaces, of different dimensions, their contributions to the eigenvalue asymptotics add up.  The same  happens if the measure has both an absolutely continuous and a singular parts.

In the usual way, the estimates and asymptotics of eigenvalues of the weighted problem, by means of the Birman-Schwinger principle, produce similar results for the number of negative eigenvalues of the Schr\"odiger-like operator as the coupling constant grows. This relation was explored in \cite{KarSh}.

Some of the results of the paper were presented in the short note \cite{RSh} without proofs.

\section{Setting and main results}\label{Setting} Since the pathbreaking papers by M.Sh. Birman and M.Z. Solomyak, it is known that it is often useful to reduce a spectral problem for a differential equation to an eigenvalue problem for a compact operator. We consider a class of compact operators encompassing the above problems, as well as their pseudodifferential versions, as particular cases. Let $\AF$ be an operator in a bounded domain $\Om\subset\R^\Nb$, such that the localization of $\AF$ to a  proper subdomain $\Om'\subset\Om$ is a pseudodifferential operator of order $-l=-\Nb/2$, up to a smoothing additive term.  The operators $(-\Delta)^{-l/2}$ and $(-\Delta+1)^{-l/2}$ in $\R^d$, restricted to $\Om$,
or $(-\D_{\Dc})^{-l/2}$, where $-\D_{\Dc}$ is the Dirichlet Laplacian in $\Om$, are typical examples of such $\AF$. By localisation we mean multiplication of $\AF$ on both sides by a smooth cut-off function $\vs\in C_0^\infty(\Om),\, \vs|_{\Om'}=1.$  Our results do not depend on the particular localisation chosen, see Section \ref{Reductions}; it is convenient to assume that the above multiplication has already been performed and the cut-offs are incorporated in $\AF$.

 Let $\Si$ be a compact Lipschitz surface  in $\Om'$, of dimension $d, \, 1\le d\le \Nb-1$ and, correspondingly, of  codimension $\dF=\Nb-d$.  Thus, locally, in appropriate coordinates, $X=(\xb;\yb) := (x_1,\dots,x_d; y_{d+1},\dots,y_{\dF})$, the corresponding piece of the surface $\Si$  is described by an equation $\yb=\vf(\xb)$, where $\vf$ is a Lipschitz vector-function. For brevity, we  describe our constructions  in a single co-ordinate neighborhood; the resulting formulae are glued together in a standard manner.  The embedding $\Si\subset\R^\Nb$ generates the surface measure $\m_\Si$ on $\Si$, $d\m_\Si = \s(\xb)d\xb,\, \s(\xb) = [\det(\mathbf{1}+(\nabla \vf)^*(\nabla\vf) )]^{\frac12}d\xb$,
 which, in turn, generates a singular measure on $\Om$, supported in $\Si$, which we also denote by $\mu_\Si,$ as long as this does not cause confusion.

Let $\m$ be a Radon measure on $\Om$. We denote by $\Mb$ its support, the smallest closed set of full measure.
The Orlicz space $L^{\Psi}(\Mb,\m), $ $\Psi(t)=(1+t)\log(1+t)-t,$ consists of $\m$-measurable functions $V$ on $\Mb$, satisfying $\int_{\Mb}\Psi(|V(X)|)d\m(X)<\infty$. For a $\m$-measurable subset $E\subset\Om$, we define the norm
 \begin{equation}\label{AvNorm}
   \|V\|_{E}^{(av,\Psi,\m)} :=\sup\left\{\left|\int_{E\cap\Mb}Vg d\m\right|:\int_{E\cap\Mb}\F(|g|)d\m\le\m(E\cap\Mb)\right\},
 \end{equation}
 if $\m(E\cap\Mb)>0,$ and $\|V\|_{E}^{(av,\Psi,\mu)}=0$ otherwise; here $\F$ is the Orlicz complementary function to $\Psi$, $\F(t)=e^t-1-t$.
Such \emph{averaged} norms, first introduced by M.Z.Solomyak in \cite{Sol94}, have played an important role in the study of the eigenvalue distribution in the critical case.
For a real-valued function $V\in L_1(\Mb,\mu)$, we consider the quadratic form
\begin{equation}\label{eq1}
    \Fb_V[u] = \Fb_V[u;\AF;\m] := \int_{\Mb} V(X)|(\AF u)(X)|^2 d\m(X), \quad u\in L_2(\Om).
\end{equation}
We will see later that, defined initially on continuous functions,  this quadratic form is bounded in $L_2(\Om),$  as soon as $V$ belongs to  $L^{\Psi}(\Mb,\m),$ and can be extended by
continuity  to the whole of $L_2(\Om);$ in this way it defines
 a bounded selfadjoint operator.
  This operator is denoted by $\Tb(V,\m)=\Tb(V,\m, \AF)$.

  The case of principal interest for us is $\mu$ being the measure $\mu_\Si$ and $\Mb$ being the surface $\Si.$  Here we will use the notation $ L^{\Psi} = L^{\Psi}(\Si)$
  for the Orlicz space,  $\|V\|_{E}^{(av,\Psi)}$ for the averaged norm and $\Tb(V)\equiv\Tb(V,\Si)\equiv\Tb(V,\Si,\AF)$ for the operator defined by the form \eqref{eq1} with $\m=\mu_\Si$.

For a compact self-adjoint operator $\Tb$ in a Hilbert space, we denote by $\la_k^{\pm}(\Tb)$ the positive (negative) eigenvalues of $\Tb$ in the non-increasing order of their absolute values, repeated according to their multiplicities. By $n_{\pm}(\la,\Tb)$ we denote the counting function of $\la_{k}^{\pm}(\Tb)$.  The notation $n(\la,\Tb)$ is used for the counting function of  singular numbers of the (not necessarily self-adjoint) operator $\Tb.$  When the operator is associated with a quadratic form $\Fb$, the notation $n_{\pm}(\la,\Fb)$ etc. is sometimes used.

Our first main result is the following eigenvalue estimate.
\begin{thm}\label{ThEstimate} Let $\Si$ be a compact Lipschitz surface of dimension $d<\Nb$ in $\Om'\subset\R^\Nb$ and $V\in L^{\Psi}(\Si)$.
Then  for the operator $\Tb=\Tb(V,\Si,\AF)$, the estimate
\begin{equation}\label{main estmate}n(\la, \Tb)\le C \|V\|^{(av, \Psi)}_{\Si}\la^{-1}
\end{equation}
holds with a constant $C$ depending on the surface $\Si$ and the operator $\AF$ but independent of the function $V$.
\end{thm}
Theorem  \ref{ThEstimate} extends to the case of singular measures  supported on surfaces in $\R^{\Nb}$ the estimates obtained by M.Z. Solomyak in \cite{Sol94} for domains (cubes) in $\R^\Nb$ for an even $\Nb$ and in \cite{SZSol} for an odd $\Nb$. In both cases, the operator $\AF_0=(1-\Delta)^{-\Nb/4}$ played the role of $\AF$. The passage to a more general $\AF$ is easy and is carried out at the end of Section \ref{EigEst}.

Theorem \ref{ThEstimate} follows from a spectral estimate in a more general setting  extending the considerations in \cite{KarSh}.
\begin{defin}\label{AhlforsDef}Let $\mu$ be a positive Radon measure on $\mathbb{R}^{\mathbf{N}}$. We say that it  is $\a$- Ahlfors regular (an $\a$-AR-measure),  $\alpha \in (0, \mathbf{N}]$,
 if there exist positive constants $c_0$ and $c_1$ such that
\begin{equation}\label{Ahlfors}
c_0r^{\alpha} \le \mu(B(X, r)) \le c_1r^{\alpha}\;
\end{equation}
for all $0< r \le \mathrm{diam(supp}\,\mu)$ and all $X\in$ $\mathrm{supp}\, \mu$, where $B(X, r)$ is the ball of radius $r$ centred at $X$ and
the constants $c_0$ and $c_1$ are independent of the balls.
\end{defin}
If  $\mu$ is an $\a$-AR-measure, then it is equivalent to the $\alpha$-dimensional Hausdorff measure on its support (see, e.g., \cite[Lemma 1.2]{Dav}).
The measure $\m_\Si$ for a compact Lipschitz surface of dimension $d$ in $\R^{\Nb}$ is, obviously, $d$-AR.

\begin{thm}\label{thm.eig.ahlf} Let $\mu$ be a compactly supported $\a$-AR measure, $0<\a\le\Nb,$ and $V\in L^{\Psi}(\mu)$. Then for the operator $\Tb(V,\mu,\AF)$,
the estimate
\begin{equation}\label{Gen.Estim} n_{\pm}(\la, \Tb(V,\mu,\AF))\le C(\a,\mu,\AF)\|V\|^{(av, \Psi)}_{\Mb}\la^{-1}
\end{equation}
holds with a constant $C(\a,\mu,\AF)$ depending on  the domain $\Om,$ the measure $\m$, and the operator $\AF$ but not on the  function $V$.
\end{thm}

To formulate the result on the eigenvalue asymptotics, we need more notation.
According to the Rademacher theorem, the function $\vf$ is differentiable  $\m_\Si$-almost everywhere. At such `regular' points $X_0$, the tangent $d$-dimensional plane $T_{X_0}\Si$  and the normal $\dF$-dimensional plane $N_{X_0}\Si$ exist. The principal symbol $a_{-l}(X,\X)$ of the operator $\AF$ can be expressed in a neighbourhood of such a point in the coordinates
$(\xb,\yb;\x,\y)$, with $\xb\in T_{X_0}\Si$, $\yb\in N_{X_0}\Si$ and the corresponding co-variables $\x,\y$; we denote it again $a_{-l}(X_0;\x,\y)$. In this notation, we set
\begin{equation}\label{eq4}
    r_{-d}(X_0,\x)= (2\pi)^{-\dF}\int_{N_{X_0}\Si}|a_{-l}(X_0;\x,\y)|^2d\y, \quad \x\in T^*_{X_0}\Si.
\end{equation}
This function is defined almost everywhere on  $T^*\Si$ and is order $-d$ homogeneous in $\x$. Now we formulate our result on eigenvalue asymptotics.

\begin{thm}\label{ThmAs}Under the  conditions of Theorem \ref{ThEstimate}, the counting function for the eigenvalues of the operator $\Tb$ has the following asymptotics
\begin{equation}\label{eq5}
    n_{\pm}(\la,\Tb(V,\Si))\sim \la^{-1}C_{\pm}(V,\Si, \AF),\, C_{\pm}=d^{-1}(2\pi)^{-d}\int_{S^*\Si}V_{\pm}(X)r_{-d}(X,\x) d\m_{\Si}(X)d\x,
\end{equation}
where the integration is performed over the cosphere bundle of $\Si$ and $V_{\pm}(X)$ denotes the positive, respectively, negative part of the function $V$.
\end{thm}

If several Lipschitz surfaces, of possibly different dimensions, are present and the measure $P$ has a possibly nontrivial absolutely continuous part, the above eigenvalue asymptotics still holds, with the coefficient being the sum of the coefficients calculated for all components of the measure.

  An important particular case of our general considerations concerns  $\AF$ being an appropriate negative power of the Laplace operator with some boundary conditions.
  More precisely, $\AF_0=(-\Delta)^{-\Nb/4}+S$,  where $S$ is an operator smoothing inside $\Om$ (typically, a singular Green operator.) In this case, the eigenvalue problem for the operator $\Tb$ can be reduced, up to negligible terms (which we disregard throughout this section), to the weighted polyharmonic eigenvalue problem
  \begin{equation}\label{Diff}
    \la(-\Delta)^l f=Pf
  \end{equation}
  understood in the distributional sense:
  \begin{equation}\label{Diff1}
    \la \int_{\Om} f (-\Delta)^l h dX=\int_{\Om} fhP, \quad h\in \Dc(\Om),
  \end{equation}
  or, for $P=V\mu_{\Si}$,
   \begin{equation}\label{Diff11}
    \la \int_{\Om} f (-\Delta)^l h dX= \int_{\Si}fhV d\mu_{\Si}, \quad h\in \Dc(\Om),
  \end{equation}
with some boundary conditions understood, again, in the distributional sense.
 If the geometry of $\Si$ is sufficiently `nice', the  spectral problem \eqref{Diff1} or \eqref{Diff11} can be expressed  more explicitly.
For example, if $\Nb=2,\, l=1$, $P=V\m_\Si$,  and $\Si$ is a Lipschitz curve \emph{inside} $\Om$, we arrive (see, e.g., \cite{Agr}) at the \emph{transmission (conjugation)} problem
  \begin{equation}\label{Diff2}
    -\Delta f=0 \ \mbox{ outside } \Si,\quad \la[f_n](X)=V(X)f(X) \mbox{ on } \Si,\qquad f\in H^1(\Om),
  \end{equation}
 where $[f_n]$ is the jump of the normal derivative $f_n$  at $\Si$.
Note that if $\Si$ \emph{is} the boundary of $\Om$, we obtain a Steklov type problem, associated with  the Neumann-to-Dirichlet operator,
\begin{equation}\label{Diff3}
    -\Delta f=0 \ \mbox{ for }X\in\Om,\quad \la f_n(X)=V(X)f(X)\ \mbox{ on } \ \Si .
  \end{equation}
This case is \emph{not} covered by the reasoning in this paper as the surface $\Si$ is \emph{not} contained in $\Om$. We stress here that the question on the eigenvalue asymptotics for the Steklov problem with Lipschitz boundary is still open, even in the two-dimensional case, while for the transmission problem with a nice weight on a Lipschitz surface the eigenvalue asymptotics was established in \cite{RT}.

Let now, still for $\Nb=2$, $\Si\subset\Om$, the measure $P$ have both absolutely continuous and singular parts, $P=V_0 dX+V_1 \m_\Si$ for a Lipschitz curve $\Si\subset\Om$, with $V_0\in L^\Psi(\Om)$ and $V_1\in L^\Psi(\Si)$. The eigenvalue problem \eqref{Diff1} takes the form
\begin{equation}\label{Diff4}
-\la\Delta f(X)=V_0(X)f(X) \ \mbox{ for }\ X\in \Om\setminus\Si, \quad  \la [f_n](X)=V_1(X)f(X) \ \mbox{ on } \ \Si.
\end{equation}
So, here the spectral parameter is present both in the differential equation and the transmission condition. Our results show that they both contribute to the leading term in the eigenvalue asymptotics. \emph{Boundary} problems of this type have been considered by A. Kozhevnikov, see \cite{Kozh}.

Let us now pass to the case $\Nb=4,$ $l=2$. Here  we have the following  choice for the dimension $d$ of $\Si:$ $d=1,2,$ or $3$. For $d=3$, $\dF=1$, we arrive again at a transmission problem similar to \eqref{Diff2}:
\begin{align}\label{Diff5}
\Delta^2 f (X)=0 \ \mbox{ for } \ X\in \Om\setminus\Si,\quad \la[(-\Delta f)_n](X)=V(X)f(X) \ \mbox{ on } \ \Si, \\
 f\in H^2(\Om). \nonumber
\end{align}
Next, for $d=2,\,\dF=2$, let $X=(\xb,\yb)\in \Om\subset\R^4$, $\xb,\yb\in\R^2$, let $\Si=\{X:\yb=0\}\cap\Om$ be (a piece of) the two-dimensional plane in $\R^4$,
and let $V=V(\xb)$. Then, after some simple calculations, we arrive at the following problem
\begin{equation}\label{Diff6}
\Delta^2f(X)=0 \ \mbox{ for } \ X\in \Om\setminus\Si, \quad \la\lim_{\de\to 0}\int_{|\yb|=\de}(-\Delta f)_n(\xb, \yb)ds(\yb)=V(\xb)f(\xb,0) ,
\end{equation}
where, for  each fixed $\xb$, the integral of the normal derivative of the Laplacian of $f$
is taken over the $\de$-circle in the $\yb$-plane.
Finally, for $d=1$, when the manifold $\Si$ is the line $\yb=0$ in the coordinates $X=(\xb,\yb), \ \xb\in\R^1,\, \yb\in\R^3$\,
in $\R^4$, the resulting problem is
 \begin{equation}\label{Diff7}
\Delta^2f(X)=0 \  \mbox{ for } \ X\in \Om\setminus\Si, \quad \la\lim_{\de\to 0}\int_{|\yb|=\de} (-\Delta f)_n(\xb, \yb)ds(\yb)=V(\xb)f(\xb,0),
\end{equation}
where, for  each fixed $\xb$, the integration is performed over the 2D sphere $|\yb|=\de$. One can interpret the conditions on the surface $\Si$ in \eqref{Diff6}, \eqref{Diff7} as
multi-dimensional versions of the transmission conditions in \eqref{Diff2}, \eqref{Diff5}. Our results show that all these spectral problems have the same order of the eigenvalue asymptotics.

If the measure $P$ has a nonzero absolutely continuous part with density $V_0$, one should replace  in \eqref{Diff5}-\eqref{Diff7} the equation
$\Delta^2f(X)=0$  with $\la\Delta^2f(X)=V_0(X)f(X)$, so, again, we arrive at spectral problems containing the spectral parameter both in the equation and in the transmission conditions.

As we already mentioned, if the support of the singular measure  consists of several disjoint surfaces, of possibly different dimensions, their contributions to the eigenvalue estimates have the same order and the coefficients in the asymptotics add up.

\begin{cor}\label{general.measure} Let a measure $P$  with a compact support in $\Om\subset \R^{\Nb}$ have the absolutely continuous part
$P_{ac}=V_0(X)dX,$ $V_0\in L^{\Psi}(\Om,dX)$ and the singular part $P_{sing}=\sum P_j$, where $P_j=V_j\m_{\Si_j}$, $\Si_j$ are disjoint Lipschitz surfaces of dimension $d_j<\Nb$,
$\m_{\Si_j}$ are the measures induced by the embeddings of $\Si_j$ into $\R^{\Nb}$,  and $V_j\in L^{\Psi}(\Si_j,\mu_{\Si_{j}})$. Then for the operator in $L_2(\Om)$ defined by the quadratic form $\Fb_{P}[u]=\int_\Om|(\AF u)(X)|^2 P$, the following asymptotic formula holds
\begin{equation}\label{As.several}
n_{\pm}(\la,\Fb )\sim \la^{-1}C(P,\AF), \,\la \to 0;
\end{equation}
here $C(P,\AF)$ is the sum of the asymptotic coefficients  in \eqref{eq5} corresponding to all surfaces $\Si_j$, plus the term coming from $P_{ac}$,\ $C(P,\AF)=C(P_{ac},\AF)+\sum C(V_j,\Si_j)$,  where
\begin{equation}\label{AC.coeff}
C(P_{ac},\AF)=\Nb^{-1}(2\pi)^{-\Nb}\int_{S^*\Om}V_{0,{\pm}}(X)|a_{-l}(X,\Xi)|^2 dX d\Xi.
\end{equation}
\end{cor}
The disjointness conditions above can be considerably relaxed, see Sect. \ref{7}.

\section{Some reductions}\label{Reductions} This section contains some technical observations that are used further on in the paper to reduce general eigenvalue estimates
and asymptotics to more convenient setting. Similar arguments for spectral problems for  differential operators are a well known part of mathematical folklore. They have been used systematically, starting with the papers by M.Sh. Birman and M.Z. Solomyak, for the past 50 years. They are usually  proved in a line or two. Our pseudodifferential versions require some additional, somewhat technical, reasoning, which nevertheless follows the classical pattern, and the results are quite natural. Readers familiar with the classical versions and not interested in these details can skip this section without detriment to understanding the rest of the paper. For brevity, we take a surface of dimension $d=\Nb$ to mean a domain in $\R^{\Nb}$. For an operator $\Tb$, we  denote by $\Db^{\pm}(\Tb)$ the quantity
\begin{equation*}
    \Db^{\pm}(\Tb) :=\limsup_{\la\to 0} \la n_{\pm}(\la,\Tb)=\limsup_{k\to\infty} \pm k\la_k^{\pm}(\Tb).
\end{equation*}

\textbf{Observation 1:} Localization 1. It is sufficient to prove the eigenvalue estimate
 for a surface being described by only one coordinate neighbourhood. Indeed,  if $\Si$ is split into a finite number $J$ of disjoint surfaces $\Si_j, j=1,\dots,J,$ of possibly different dimensions,
 and $P_j=V_j\m_{\Si_j}$ is the restriction of $P$ to $\Si_j$, then $\Db^{\pm}(\Tb(V,\Si,\AF))\le J\sum_j \Db^{\pm}(\Tb(V_j,\Si_j,\AF))$. This property follows from the Ky Fan inequality for the sum of operators,
$\Tb(V,\Si,\AF)=\sum_j\Tb(V_j,\Si_j,\AF)$.

\textbf{Observation 2:}  {Lower order terms.} Let $\BFB$ be a pseudodifferential operator of order $\be<-l$. Then  $\Db^{\pm}(\Tb(V,\Si,\BFB))=0$. This follows from the (already mentioned) result in \cite{BS} that a finite singular measure gives a lower order contribution  to the  eigenvalue distribution for higher order problems.

\textbf{Observation 3:} {Perturbations by lower order terms.}\label{pert.lower} Let $\BFB$ be as above. Then
\begin{equation}\label{add.pertur}
     \Db^{\pm}(\Tb(V,\Si,\AF+\BFB))= \Db^{\pm}(\Tb(V,\Si,\AF)).
\end{equation}
This follows from the inequality
$$
(1-\e)|a|^2 + \left(1-\frac{1}{\e}\right) |b|^2 \le|a+b|^2\le (1+\e)|a|^2 + \left(1+\frac{1}{\e}\right) |b|^2
$$
with $a=(\AF u)(X),\, b=(\BFB u)(X)$ and  the previous observation.

\textbf{Observation 4:} {Localization 2.}\label{loc1} Let $\Om_1$ be  an open subset in $\Om$ such that $\Si\cap\overline{\Om_1}=\varnothing$,
and let $V\in L^{\Psi}(\Si)$. Then for  the eigenvalues of the  operator $\Tb(V,\Om_1)$ on $L_2(\Om_1)$ defined by the form $\Fb_{V,\Om_1}[u]=\int_\Si V(X)|(\AF u)(X)|^2d\m,$ $u\in L_2(\Om_1)$, one has the following estimate $n_{\pm}(\la, \Tb(V,\Om_1))=o(\la^{-1})$ as $\la\to 0.$

\emph{Proof.} Let $\hi \in C_0^\infty$ be a smooth function equal to 0 in a neighborhood of $\overline\Om_1$ and to $1$ in a neighborhood of $\Si$. Then
\begin{equation}\label{cutoff} \Fb_{V,\Om_1}[u]=\int_\Si V(X)|(\hi(X)\AF u)(X)|^2d\m=\int_\Si V(X)|([\hi,\AF]u)(X)|^2d\m.
\end{equation}
The commutator  $[\hi,\AF ]$ is a pseudodifferential operator  of order $-l-1$, and therefore, by Observation 2, the  eigenvalues  of the corresponding operator $\Tb(V,\Om_1)$ decay faster than $k^{-1}$.

It follows, in particular, that the eigenvalue counting function gets a lower order perturbation if we perturb the operator $\AF$ outside a neighborhood of the surface $\Si$. In particular, this gives us freedom in choosing cut-off functions away from $\Si$ or adding operators smoothing away from the boundary -- the possibility already mentioned.

The next, more complicated, statement is used in the study of eigenvalue asymptotics. It says that if two measures have supports separated by a positive distance, then, up to a lower order term, the counting functions behave additively with respect to the measures. This includes the important case when  absolutely continuous measures are present. If both measures are absolutely continuous, this is a classical fact.

\begin{lem}{(Localization 3)}\label{LemLocalization} Let $P=P_1+P_2$, where $P_j=V_j \m_{\Si_j}$ is a measure supported on a compact Lipschitz surface $\Si_j$ of dimension $d_j\in[1,\Nb]$, $j = 1, 2$ (the cases $d_1=\Nb$ and $ d_2=\Nb$ correspond to $\Si_1 $, respectively $\Si_2$, being  domains in $\Om\subset\R^{\Nb}$ and the  measures being absolutely continuous with respect to the Lebesgue measure). Suppose that $\dist(\supp P_1, \supp P_2)>0$. Then
\begin{equation}\label{Localization}
    n_{\pm}(\la,\Tb(P_1+P_2))=n_{\pm}(\la, \Tb(P_1))+n_{\pm}(\la, \Tb(P_2))+o(\la^{-1})\ \mbox{ as }\ \la\to 0.
\end{equation}
\end{lem}
\begin{proof} Consider two disjoint open sets $\Om_1,\Om_2\subset \Om$, such that $\overline{\Si_j}\subset\Om_j$ and $\overline{\Om_1}\cup\overline{\Om_2}\supset\Om$.
Every function $u\in L_2(\Om)$ splits into the (orthogonal) sum $u=u_1\oplus u_2$, $u_j\in L_2(\Om_j)$. The quadratic form of the operator $\Tb(P_1+P_2)$ splits as follows
\begin{align}\label{splitting}
&  \Fb(P_1+P_2)[u] :=  (\Tb(P_1+P_2)u,u)_{L_2(\Om)} \\ \nonumber & =\int_\Om V_1(X)|\AF(u_1)(X)+ \AF(u_2)(X)|^2 d\m_{\Si_1}
+  \int_\Om V_2(X)|\AF(u_1)(X)+ \AF(u_2)(X)|^2d\m_{\Si_2}\\ \nonumber
& = \Fb_1(P_1)[u_1]+\Fb_2(P_2)[u_2]+\Fb_R[u_1,u_2]\\
\nonumber & :=\int_\Om V_1(X)|\AF(u_1)(X)|^2 d\mu_{\Si_1}+\int_\Om V_2(X)|\AF(u_2)(X)|^2d\mu_{\Si_2}+\Fb_R[u_1,u_2],
\end{align}
where  the remainder term $ \Fb_R[u_1,u_2]$ is a form in functions $u_j\in L_2(\Om_j)$ with the following property: if a term in $\Fb_R$ contains $V_j$, then it necessarily contains
$u_{3-j}$, so it always contains a measure and a function with disjoint supports. If such a term has the form $\int_{\Om} V_1|\AF u_2|^2 d\mu_{\Si_1}$, the corresponding operator $\Tb$ satisfies $n_{\pm}(\la, \Tb)=o(\la^{-1})$ by Observation 4. If, on the other hand, such term has the form $\int_{\Om} V_1(\AF u_1)\overline{(\AF u_2)}d\m_{\Si_1}$, then by the Schwartz inequality,
 $$\left|\int_{\Om} V_1(\AF u_1)\overline{(\AF u_2)}d\m_{\Si_1}\right|\le \left(\int_{\Om} |V_1||\AF u_1|^2d{\mu_{\Si_1}}\right)^{1/2}\left(\int_{\Om} |V_1||\AF u_2|^2d{\mu_{\Si_1}}\right)^{1/2},$$
 and the last factor, again, provides the required $o$ estimate for the eigenvalues.
Now we observe that the forms $\Fb_1(P_1)[u_1],\,\Fb_2(P_2)[u_2]$ in \eqref{splitting} act on orthogonal subspaces. Then the spectrum of the sum of the corresponding operators $\Tb_1,\Tb_2$  equals the union of the spectra of the summands, and hence $n_{\pm}(\la,\Tb_1+\Tb_2)=
n_{\pm}(\la,\Tb_1)+n_{\pm}(\la,\Tb_2).$ Since the term $\Fb_R$ in  \eqref{splitting} makes a weaker contribution,
\begin{equation}\label{splitting1}
n_{\pm}(\la,\Tb(P_1+P_2))=
n_{\pm}(\la,\Tb_1)+n_{\pm}(\la,\Tb_2)+o(\la^{-1}).
\end{equation}
Now consider the operator $\Tb(P_1)$. It has the quadratic form
 \begin{equation*}
 \Fb(P_1)[u]=\int V_1(X)|\AF(u_1\oplus u_2)|(X)d\mu_{\Si_1}.
 \end{equation*}

  Similarly to \eqref{splitting}, we represent it as
\begin{equation}\label{splitting2}\Fb(P_1)[u]=\int_\Om V_1(X)|\AF(u_1)(X)|^2 d\mu_{\Si_1}+\Fb_{R_1}[u_1,u_2],
\end{equation}
with $\Fb_{R_1}$ having the same structure as $\Fb_{R}$ in \eqref{splitting}. Again, the form $\Fb_{R_1}$ generates an operator with eigenvalues decaying faster than $k^{-1}$,
and we obtain
\begin{equation}\label{splitting3}n_{\pm}(\la, \Tb(P_1))=n_{\pm}(\la,\Tb_1)+o(\la^{-1}).
\end{equation}
In the same way,
\begin{equation}\label{splitting4}n_{\pm}(\la, \Tb(P_2))=n_{\pm}(\la,\Tb_2)+o(\la^{-1}).
\end{equation}
Finally, we substitute \eqref{splitting3}, \eqref{splitting4} into \eqref{splitting1} to obtain \eqref{Localization}.
\end{proof}

The following corollary of Lemma \ref{LemLocalization} allows one to separate the positive and the negative parts of the function $V$ when studying the distribution of the positive and the negative eigenvalues of $\Tb(V,\Si,\AF)$.
\begin{cor}\label{cor.plusminus} Let $\Si$ be a Lipschitz surface and $V\in L^{\Psi}(\Si)$. Let $\Si_{\pm}$  be relatively open subsets of $\Si$ such that
$V_{\pm}\ge 0$ in $\Si_{\pm}$, $V=0$ in $\Si\setminus(\Si_+\cup\Si_-)$, and $\dist(\overline{\Si_+},\overline{\Si_-})>0$.
Then
\begin{equation}\label{sign.separation}
n_{\pm}(\la, \Tb(V,\Si))=n_+(\la, \Tb(V_\pm, \Si))+o(\la^{-1}) \ \mbox{ as }\ \la\to 0.
\end{equation}
\end{cor}
In other words, up to a lower order remainder, the behavior of the positive, respectively negative, eigenvalues of the operator $\Tb(V)$ is determined by the positive, respectively negative, part of the density $V$. To prove this property, we can use \eqref{splitting1}, taking as $P_1$ the restriction of the measure $P$ to the set $\Si_+$, and as $P_2$ its restriction to $\Si_-$, and recall that $n_-(\la, \Tb(V,\Si_+))=n_+(\la, \Tb(V,\Si_-))=0$.

\section{Geometry considerations}
Our proof of Theorems \ref{ThEstimate} and \ref{thm.eig.ahlf}  relies upon certain geometric observations that might be of an independent interest.

Let $\mathcal{A}\subset\R^{\Nb}$ be a $k$-dimensional affine subspace, $0 \le k < \mathbf{N}$, $\mathcal{A} = a + \mathcal{L},$
where $a \in \mathbb{R}^\mathbf{N}$ and $\mathcal{L}$ is a $k$-dimensional linear subspace in $\mathbb{R}^\mathbf{N}$. In the case $k = 0$,
$\mathcal{L}$ is a $0$-dimensional linear subspace, i.e $\mathcal{L} = \{0\}$, and $\mathcal{A}$ is just the singleton $\{a\}$.
The polar plane of $\Ac$ is the $\mathbf{N} - k$ dimensional linear subspace of $\mathbb{R}^\mathbf{N},$
\begin{equation}\label{perp}
\mathcal{A}^\perp := \mathcal{L}^\perp = \{Y \in \mathbb{R}^\mathbf{N} : \ (X, Y) = 0 \ \mbox{ for all } \ X \in  \mathcal{L}\} ,
\end{equation}
where $(\cdot, \cdot)$ denotes the standard inner product in $\mathbb{R}^\mathbf{N}$.
  We will say that $\mathcal{A}$ is orthogonal to
a vector $b \in \mathbb{R}^\mathbf{N}\setminus\{0\}$ if $b \in \mathcal{A}^\perp$.
For $k = \mathbf{N} - 1$,  $\mathcal{A}^\perp$ is the one-dimensional linear subspace spanned by $b$.
For a linear subspace $\mathcal{M}$ in $\mathbb{R}^\mathbf{N}$, we denote by $\Mc_\mathbb{S}$ the trace of $\Mc$ on  $\mathbb{S}^{\Nb-1}$,
$\mathcal{M}_{\mathbb{S}} := \mathcal{M}\cap\mathbb{S}^{\mathbf{N} - 1} ,$
where $\mathbb{S}^{\mathbf{N} - 1}$ is the unit sphere in $\mathbb{R}^\mathbf{N}$, $\mathbb{S}^{\mathbf{N} - 1} := \{X \in \mathbb{R}^\mathbf{N} : \ |X| = 1\} .$

\begin{lem}\label{geom}
Let $\WF$ be an at most countable family of proper linear subspaces of $\mathbb{R}^\mathbf{N}$. Then there exists
an orthonormal basis $\eb_1, \dots, \eb_\mathbf{N}$ of $\mathbb{R}^\mathbf{N}$ such that
$$
\eb_j \not\in \cup_{\mathcal{M} \in \WF} \mathcal{M} , \quad j = 1, \dots, \mathbf{N} .
$$
\end{lem}
\begin{proof}
The proof is by induction on $\mathbf{N}$. There is nothing to prove if $\mathbf{N} = 1$ since the only proper linear subspace of $\mathbb{R}$ is $\mathcal{M} = \{0\}$.
Suppose that the statement is true for $\mathbf{N} = \Nb_0$ and let us prove it for $\mathbf{N} = \Nb_0+1$. Let $\WF'$ be the (possibly empty) subset of $\WF$ consisting of
all $\Nb_0$-dimensional $\mathcal{M} \in \WF$. For every such $\mathcal{M}$, \, $\mathcal{M}^\perp$ is $1$-dimensional and $\mathcal{M}^\perp_\mathbb{S}$
consists of two points. Since $\cup_{\mathcal{M} \in \WF} \mathcal{M}_\mathbb{S}$
is an at most countable union of spheres of dimension at most $\Nb_0- 1$, the set
$$
\Theta := \left(\cup_{\mathcal{M} \in \WF'} \mathcal{M}^\perp_\mathbb{S}\right) \bigcup \left(\cup_{\mathcal{M} \in \WF} \mathcal{M}_\mathbb{S}\right)
$$
has $\Nb_0$-dimensional Lebesgue measure equal to 0. So, $\Theta \not= \mathbb{S}^{\Nb_0}$. Take any vector $\eb_{\Nb_0 + 1} \in \mathbb{S}^{\Nb_0}\setminus \Theta$
and let $\mathcal{M}_0$ be the $\Nb_0$-dimensional linear subspace of $\mathbb{R}^{\Nb_0 + 1}$ orthogonal to $\eb_{\Nb_0 + 1}$. Since
$$\eb_{\Nb_0 + 1} \not\in \cup_{\mathcal{M} \in \WF'} \mathcal{M}^\perp_\mathbb{S} ,$$
$\mathcal{M}_0$ does not coincide with any element of $\WF'$, and hence the dimension of $\mathcal{M}_0\cap\mathcal{M}$ is at most $\Nb_0 - 1$ for every
$\mathcal{M} \in \WF$. Then, by the inductive assumption, there exists an orthonormal basis $\eb_1, \dots, \eb_{\Nb_0}$ of $\mathcal{M}_0$ such that
$$
\eb_j \not\in \cup_{\mathcal{M} \in \WF} (\mathcal{M}_0\cap\mathcal{M})
= \mathcal{M}_0\cap\left(\cup_{\mathcal{M} \in \WF} \mathcal{M}\right), \quad j = 1, \dots, \Nb_0 .
$$
It is clear that $\eb_1, \dots, \eb_{\Nb_0}, \eb_{\Nb_0 + 1}$ is an orthonormal basis of $\mathbb{R}^{\Nb_0 + 1}$ and
$$
\eb_j \not\in \cup_{\mathcal{M} \in \WF} \mathcal{M} , \quad j = 1, \dots, \Nb_0 + 1 .
$$
\end{proof}

\begin{lem}\label{l1}
Let $\mu$  be a $\sigma$-finite Borel measure on $\mathbb{R}^\mathbf{N}$ such that
$\mu(\mathcal{A}) = 0$ for every $(k - 1)$-dimensional affine subspace $\Ac \subset \mathbb{R}^\mathbf{N}$ for some $k\in[1, \mathbf{N} - 1]$.
Then the set $\XF_k$ of $k$-dimensional affine subspaces $\mathcal{E}$ of $\mathbb{R}^\mathbf{N}$ such that
$\mu(\mathcal{E}) > 0$ is at most countable.
\end{lem}
\begin{proof}
The proof is similar to that of \cite[Lemma 2.13]{KarSh}.
It is sufficient to prove the lemma for finite measures as the general case then follows easily from the assumption that
$\mu$ is $\sigma$-finite.
Suppose that $\XF_k$ is uncountable.
Then there exists a $\delta > 0$ such that the set
$$
\XF_{k,\delta}:= \left\{\mathcal{E} \in \XF_k\;:\;  \mu(\mathcal{E}) > \delta\right\}
$$
is infinite. Otherwise, $\XF_k = \underset{m\in\mathbb{N}}\cup \XF_{k, \frac{1}{m}}$ would have been at most countable. Now take
distinct $\mathcal{E}_1,..., \mathcal{E}_j,... \in \XF_{k, \delta}$. We have
$\mu\left(\mathcal{E}_j\right) > \delta \quad\mbox{for all}\quad j\in\mathbb{N}\,.$

Since $\mathcal{\Ec}_j\cap \mathcal{\Ec}_{j'} ,\;j\neq j',$ is an affine subspace of dimension at most $k - 1$,
 $$
 \mu\left(\underset{j \neq j'}\cup (\mathcal{E}_j\cap \mathcal{E}_{j'})\right) = 0.
 $$
Let
$$
\tilde{\mathcal{E}}_j:= \mathcal{E}_j\setminus\underset{j' \neq j}\cup (\mathcal{E}_{j'}\cap \mathcal{E}_j)\, .
$$
Then $\tilde{\mathcal{E}}_j\cap \tilde{\mathcal{E}}_{j'}  = \varnothing,\;j \neq j'$ and $\m(\tilde{\Ec}_j)=\m(\Ec_j)$.
So, since $\m$ is finite,
$$
\sum_{j\in\mathbb{N}} \mu\left(\tilde{\mathcal{E}}_j \right) =
\mu\left(\underset{j\in\mathbb{N}}\cup \tilde{\mathcal{E}}_j  \right)  < \infty\,.
$$
On the other hand,
$\mu\left(\tilde{\mathcal{E}}_j \right) = \mu\left(\mathcal{E}_j \right) > \delta$
implies that
$\sum_{l\in\mathbb{N}} \mu_k\left(\tilde{\mathcal{E}}_j\right) \geq \underset{j\in\mathbb{N}}\sum\delta = \infty$.
This contradiction means that $\XF_k$ is at most countable.
\end{proof}
We arrive at our main geometric statement, considerably more general than what we actually need here.
\begin{thm}\label{mm}
Let $\mu$ be a  $\sigma$-finite Borel measure on $\mathbb{R}^\mathbf{N}$ without point masses.
Then there exists an orthonormal basis $\eb_1, \dots, \eb_\mathbf{N}$ of $\mathbb{R}^\mathbf{N}$ such that
$\mu(\mathcal{E}) = 0$ for every \hbox{$(\mathbf{N}-1)$}-dimensional affine subspace $\mathcal{E}$ of $\mathbb{R}^\mathbf{N}$ orthogonal to an element of this basis.
\end{thm}
\begin{proof} Let us first explain the idea of the proof.
By Lemma \ref{l1} for $k=1$, there is an at most countable set of one-dimensional affine subspaces whose $\m$-measure is positive. Subtracting the portion of measure $\m$ living on all these subspaces and applying Lemma \ref{l1} to the resulting measure, this time with $k=2$, we conclude that at most  countably many 2-dimensional affine subspaces are charged positively. We subtract the part of our measure living on these subspaces, and so on. This procedure is repeated in all dimensions, after which it turns out that the remaining measure is zero on \emph{all}  affine subspaces in  $\R^{\N}$, and the proof is completed by applying Lemma \ref{geom}.

Now the formal proof. Denote $\m=\m_1$.
 Lemma \ref{l1} with $k = 1$  implies that the set $\XF_1$ of all $1$-dimensional affine subspaces $\mathcal{E}$ of $\mathbb{R}^\mathbf{N}$ such that
$\mu_1(\mathcal{E}) > 0$ is at most countable. We introduce

$\Zc_1 := \cup_{\mathcal{E} \in \XF_1} \mathcal{E}$,\,
$\mu_1^0(E) := \mu_1(E\cap\Zc_1)$ {for every Borel set} $E \subseteq \mathbb{R}^\mathbf{N},$ { and}
$ \mu_2 := \mu_1 - \mu_1^0. $

Measure $\mu_2$ annuls every  $1$-dimensional affine subspace $\mathcal{E}$. Indeed,
\begin{align*}
 & \mu_2(\mathcal{E}) \le \mu_1(\mathcal{E}) = 0 \quad\mbox{if}\quad \mathcal{E} \not\in  \XF_1 , \\
& \mu_2(\mathcal{E}) = \mu_1(\mathcal{E}) - \mu_1^0(\mathcal{E}) = \mu_1(\mathcal{E}) - \mu_1(\mathcal{E}) = 0 \quad\mbox{if}\quad \mathcal{E} \in  \XF_1 .
\end{align*}
Therefore  we can apply Lemma \ref{l1} again, this time with $k = 2$ and $\m=\m_2$.

Repeating this procedure, we obtain at the $k$-th step a $\sigma$-finite  measure $\mu_k$ on $\mathbb{R}^\mathbf{N}$ such that
$\mu_k(\mathcal{A}) = 0$ for every $(k - 1)$-dimensional affine subspace $\mathcal{A}$ of $\mathbb{R}^\mathbf{N}$ and
the set $\XF_k$ of all $k$-dimensional affine subspaces $\mathcal{E}$ of $\mathbb{R}^\mathbf{N}$ with
$\mu_k(\mathcal{E}) > 0$ is at most countable. Similarly to the above, we set
$ \Zc_k := \cup_{\mathcal{E} \in \XF_k} \mathcal{E} ,$
$\mu_k^0(E) := \mu_k(E\cap\Zc_k)$  {for every Borel set} $E \subseteq \mathbb{R}^\mathbf{N}$,  { and}
$ \mu_{k + 1} := \mu_k - \mu_k^0 .$
Then, similarly to the above, $\mu_{k + 1}(\mathcal{E}) = 0$ for every $k$-dimensional affine subspace $\mathcal{E}$.

For every $\mathcal{E} \in \XF_k$, \ $\mathcal{E}^\perp$ is an $(\mathbf{N} - k)$-dimensional linear subspace of $\mathbb{R}^\mathbf{N}$, $1 \le k \le \mathbf{N} - 1$.
By Lemma \ref{geom}, there exists an orthonormal basis $\eb_1, \dots, \eb_\Nb$ of $\mathbb{R}^\mathbf{N}$ such that
\begin{equation}\label{jnotin}
\eb_j \not\in \pmb{\pmb{\Xi}} := \cup_{k = 1}^{\mathbf{N} - 1} \cup_{\mathcal{E} \in \XF_k} \mathcal{E}^\perp , \quad j = 1, \dots, \Nb .
\end{equation}

Take any $j = 1, \dots, \mathbf{N}$. Let $\mathcal{E}_0$ be an $(\mathbf{N} - 1)$-dimensional affine subspace orthogonal to $\eb_j$. Then
for any $k = 1, \dots, \mathbf{N} - 1$,
\begin{equation}\label{knotin}
\mathcal{E} \in \XF_k \quad\Longrightarrow\quad \mathcal{E} \not\subseteq \mathcal{E}_0 .
\end{equation}
Indeed, if $\mathcal{E} \subseteq \mathcal{E}_0$, then
$$
\eb_j \in \mathcal{E}_0^\perp \subseteq \mathcal{E}^\perp \subseteq \pmb{\pmb{\Xi}} ,
$$
which contradicts \eqref{jnotin}.

Since $\mathcal{E}_0 \not\in \XF_{\Nb - 1}$, \ $\mu_{\Nb - 1}(\mathcal{E}_0) = 0$. For any $\mathcal{E} \in \XF_k$, \eqref{knotin} implies that
$\mathcal{E}_0\cap\mathcal{E}$ is an affine subspace of dimension at most $k - 1$. Then $\mu_k(\mathcal{E}_0\cap\mathcal{E}) = 0$ and
$$
\mu_k^0(\mathcal{E}_0) = \mu_k(\mathcal{E}_0\cap\Zc_k) \le \sum _{\mathcal{E} \in \XF_k} \mu_k(\mathcal{E}_0\cap\mathcal{E}) = 0 .
$$
Hence
\begin{align*}
0 & = \mu_{\Nb - 1}(\mathcal{E}_0) = \mu_{\Nb - 2}(\mathcal{E}_0) - \mu_{\Nb - 2}^0(\mathcal{E}_0) = \mu_{\Nb - 2}(\mathcal{E}_0) =  \mu_{\Nb - 3}(\mathcal{E}_0) \\
& = \cdots =  \mu_1(\mathcal{E}_0) =  \mu(\mathcal{E}_0) .
\end{align*}
\end{proof}

 \section{Estimates}\label{estimates}
 The method of deriving eigenvalue estimates for non-smooth spectral problems developed by M.Sh. Birman and M.Z. Solomyak in the late 60's has remained since then the
 most efficient approach to such singular problems. The method is based on constructing piecewise polynomial approximations of functions in Sobolev spaces.
We follow the presentation in \cite{Sol94}, with necessary modifications caused by the presence of singular measures.
\subsection{A homogeneous  H\"older inequality}
We recall that for a bounded domain $G \subseteq \mathbb{R}^{\mathbf{N}}$, the Sobolev norm and the homogeneous Sobolev semi-norm are defined by
\begin{align*}
\|u\|^2_{H^{s}(G)} := \sum_{|\n|_1 \le s} \left\|\partial^\n u\right\|^2_{L_2(G)} , \qquad
\|u\|^2_{H^{s}_{\mathrm{hom}}(G)} := \sum_{|\n|_1 = s} \left\|\partial^\n u\right\|^2_{L_2(G)}
\end{align*}
for an integer  $s $, and by
\begin{align*}
& \|u\|^2_{H^{s}(G)} := \|u\|^2_{H^{m}(G)} +\|u\|^2_{H^{s}_{\mathrm{hom}}(G)}, \\
& \|u\|^2_{H^{s}_{\mathrm{hom}}(G)} := \sum_{|\n|_1 = m} \int_G\int_G \frac{\left|(\partial^\n u)(X) - (\partial^\n u)(Y)\right|^2}{|X - Y|_2^{\mathbf{ N} + 1
}}\, dXdY
\end{align*}
for a half-integer $s =m+\frac12$. (Here $|\cdot|_q$ denotes the standard norm in $\ell^q$.) The space  $\overset{\circ}{H}{}^s(G)$, $s>\frac12$, is defined as the closure of $C^\infty_0(G)$ in $H^s$ norm which is equivalent on $\overset{\circ}{H}{}^s(G)$ to the homogeneous Sobolev norm.
The semi-norm $\|\cdot\|^2_{H^s_{\mathrm{hom}}(G)}$ possesses the homogeneity property
$\|u(R\cdot)\|_{H^s_{\mathrm{hom}}(G)} =  R^{s - \mathbf{N}/2} \|u\|_{H^{s}_{\mathrm{hom}}(RG)}$ for all $R > 0$. Our case of interest is $s=l=\frac\Nb{2}$; here the homogeneous semi-norm is invariant under dilations.  Although this semi-norm
is not, in general, invariant under rotations, it is easy to see that under rotations, this semi-norm transforms to an equivalent one, and there exist positive constants $m_\mathbf{N}, M_\mathbf{N}$ depending only
on $\mathbf{N}$ and such that
\begin{align}\label{invar}
m_\mathbf{N} \left\|u\big(\Zb(\cdot)\big)\right\|_{H^{l}_{\mathrm{hom}}(G)}  \le \|u\|_{H^l_{\mathrm{hom}}(\Zb(G))} \le
M_\mathbf{N} \left\|u\big(\Zb(\cdot)\big)\right\|_{H^{l}_{\mathrm{hom}}(G)}  , \\
\forall u \in H^{l}(G) \nonumber
\end{align}
for every affine transformation $\Zb$ of the form $X=\Zb(Y) = RUY + X_0$, with $X_0 \in \mathbb{R}^\mathbf{N}$, $R > 0$, and an orthogonal matrix  $U \in \mathrm{O}(\Nb).$

Recall that $\Psi(t)= (1 + t)\ln(1 + t) - t$ and  $\Phi(t)= e^t - 1 - t$ are complementary Orlicz functions.
\begin{lem}\label{meascor}{\rm  (cf.\;\cite[Corollary 11.8/2]{MazBook})}
Let $G\subset\mathbb{R}^{\mathbf{N}}$ be a bounded domain with Lipschitz boundary. If a positive Borel measure $\mu$ on $\overline{G}$ satisfies, for some $\a > 0,$
\begin{equation}\label{ball}
\mu(B(X, r)) \le Kr^{\a}\;,\;\;\forall X\in\overline{G} \;\;\; \textrm{and}\;\;\;  \forall r \in \left(0, \frac12\right)\,,
\end{equation}
then the inequality

\begin{equation}\label{Trace Embedding}
\|w^2\|^{\Phi, \mu}_{\overline{G}} \le A_1 \|w\|^2_{H^{l}(G)} , \ \ \ \forall w \in H^{l}(G)\cap C(\overline{G})
\end{equation}
holds with a constant $A_1=A_1(G, \a, K).$
\end{lem}
\begin{proof}
The proof relies on \cite[Theorem 11.8]{MazBook} and is almost identical to that of \cite[Lemma 5.2]{KarSh}.
 \end{proof}
 Inequality \eqref{Trace Embedding} implies that the embedding ${H^{l}(G)}\cap C(\overline{G})$ into the exponential Orlicz space $L^{\Phi, \mu},$ defined initially on continuous functions, extends by continuity to the whole ${H^{l}(G)}.$ This continuation will be assumed already performed further on.
The boundedness of the trace operator $\G: H^{l}(G)\to L_2(\Mb,\mu),$ $f\mapsto f|_\Mb,$ used at least twice here,  follows from the embedding $L^{\Phi,\mu}(\Mb)\subset L_1(\Mb,\m)$ (recall that $\Mb$ is the support of the measure $\m$.)

We will also need the following version of the Poincar\'e inequality that can be traced back to S.L. Sobolev (a very detailed proof  can be found in
\cite[Theorem 2.4]{SZSol} for $G$ being a cube).
For every  bounded set $G\subset\mathbb{R}^{\mathbf{N}}$ with Lipschitz boundary
and every $s > 0$, there exists a constant $C_s(G) > 0$ such that
\begin{equation}\label{Poinc}
\|u\|_{H^{s}(G)} \le C_s(G) \|u\|_{H^{s}_{\mathrm{hom}}(G)}
\end{equation}
for all $u \in H^{s}(G)$ orthogonal in $L_2(G)$ to every polynomial of degree strictly less than $s$. We will denote
the optimal constant in \eqref{Poinc} with $s = \mathbf{N}/2$ by $C(G) := C_{\mathbf{N}/2}(G)$.

Our eigenvalue estimates rely on the following inequality.
\begin{lem}\label{Measlem4}Let $Q\subset \R^\Nb$ be a cube and $\mu=\mu_\Si$.
\begin{equation}\label{maz7}
\int_{\overline{Q}} V(X)|f(X)|^2d\mu(X) \le A_2  \|V\|^{(av,\Psi,\mu)}_{\overline{Q}} \|f\|_{H^{l}_{\mathrm{hom}}(Q)}^2
\end{equation}
for all $f\in H^l(Q)$ orthogonal in $L_2(Q)$ to every polynomial of degree strictly less than $l$, with constant $A_2$ depending only on $\Nb,$ the constants $c_1,c_2$, and the exponent $\a$ in \eqref{Ahlfors}.
\end{lem}
\begin{proof} If $Q=Q_1$ is a unit cube, this is an immediate generalization of
 the proof of \cite[Lemma 5.3]{KarSh}. For a cube of an arbitrary size, one should make a dilation (compression)  in \eqref{maz7}  of $Q$ to $Q_1$. Under this transformation, the homogeneous norm in
 \eqref{maz7} is invariant. As for the norm $\|V\|^{(av,\Psi,\mu)}_{\overline{Q}}$, it is not invariant under dilations, but is `almost' invariant in the sense that the dilated (compressed) norm is majorated by the initial one, with constant depending on $c_0,c_1$ and $\alpha$ in \eqref{Ahlfors}, independently of the dilation coefficient. Estimate  \eqref{maz7} is proved in \cite[Lemma 5.4]{KarSh}, and the proof does not depend on dimension  and uses nothing but the Ahlfors regularity of $\m$, so it applies to our case without changes.
\end{proof}

\subsection{Coverings and piecewise polynomial approximations}
The construction to follow is an adaptation of the method developed by M.Z. Solomyak in \cite{Sol94}. It stems from the approach initiated by M.Sh. Birman, M.Z. Solomyak, with a contribution by G.Rozenblum, in the late 60's--early 70's, see the nice exposition in \cite{BirSolLect}. Recently, this construction was used again for obtaining eigenvalue estimates for the weighted
(poly-)harmonic operator in the critical case, see \cite{KarSh}, \cite{SZSol}. We present the main idea and the structure of the proof first, and then fill in the required details.
\begin{thm}\label{BSESt} Let $\m$ be a Borel measure in $G\subset \R^\Nb$ satisfying \eqref{Ahlfors} with $\a>0$ in a bounded domain $G \subset \R^{\Nb}$ and $V\in L^{\Psi}(\Mb,\m)$. Denote by $\overset{\circ}\Tb(V,\m)$ the operator defined  in $\overset{\circ}{H}{}^{l}(G)$ by the quadratic form $\int V(X)|f(X)|^2d\m$, $f\in \overset{\circ}{H}{}^{l}(G)$, with the norm $\|f\|_{H^{l}_{\mathrm{hom}}(G)}$ in the latter space. Then
\begin{equation}\label{BS.estim}n_{\pm}(\la,\overset{\circ}\Tb(V,\m) )\le C \la^{-1}\|V\|^{(av,\Psi,\mu)}_{G},
\end{equation}
with a constant $C$ independent of $V$.
\end{thm}
\begin{proof} For $\a=\Nb$ and $\m$ being the Lebesgue measure, this result has been proved   in \cite{Sol94} for even $\Nb$ and in \cite{SZSol} for odd $\Nb$.
So, let $\a<\Nb.$ It suffices to consider the case $V\ge0.$ We use one of the possible formulations of the variational principle for compact self-adjoint operators. If $\Tb\ge0$ is such an operator on a Hilbert space $\Hc$, with the quadratic form $\tb[f]$, then for the eigenvalue counting function $n(\la,\Tb)$,
\begin{equation}\label{variation}
    n(\la,\Tb)= \min\codim\{\Yc\subset\Hc:  \tb[f]\le \la\|f\|^2, f\in\Yc\}.
\end{equation}
Here the codimension $\codim{\Yc}$ is the number of linearly independent functionals that have $\Yc$ as their common null space. Equality \eqref{variation} hints at how to prove eigenvalue estimates. Let us, for some $\la>0$, construct  \emph{some} subspace $\Yc=\Yc(\la)$ on which the inequality in \eqref{variation} holds. Then the quantity $n(\la,\Tb)$ is not greater than the codimension of the subspace we constructed, $n(\la,\Tb)\le \codim \Yc$. The more efficient we are in constructing $\Yc$, the sharper the estimate is.

Let $\Hc$ be the Sobolev space $\overset{\circ}{H}{}^l(G).$ In constructing the subspace $\Yc$, we take some finite covering $\U=\U(\la)$ of $G$ by cubes. With each cube $Q\in\U$, we associate a set of functionals, the $L_2(Q)$-scalar products with polynomials of degree less than $l$. Thus, with each cube, $\db(\Nb,l)$, the dimension of the space of polynomials of interest,
functionals are associated, altogether $|\U|\db(\Nb,l)$ of them, so, the common null space $\Yc=\Yc[\U]$ of these functionals has just this codimension (or less).
 Let $\La(\U)$ be defined as $\sup_{f\in \Yc[\U]}\tb[f]/{\|f\|^2}.$ By the variational principle \eqref{variation},
 \begin{equation}\label{ineq.codim}
    n(\La(\U),\Tb)\le |\U|\db(\Nb,l).
 \end{equation}
 So,  given $\la>0$, we need to construct a covering $\U=\U(\la)$ such that $\La(\U)\le \la$, and at the same time, $|\U|$ should be under control, and its value will produce the required estimate via \eqref{ineq.codim}.

  In our case, the quadratic form defining our operator is $\tb[f]=\int_{G} V(X)|f(X)|^2 d\mu$, with $V\ge0$ (we identify measure $\m$ with its natural extension by zero to the whole of $G$).
 The same integral over a cube $Q$ will be denoted by $\tb_{Q}[f]$. If $f\in \Yc[\U]$, then on each cube $Q$ of the covering $\U$, the restriction $f_{Q}$ of $f$ to $Q$ is orthogonal to polynomials of degree less than $l$. Suppose that for \emph{such} functions $f_Q$   an estimate of the form
 \begin{equation}\label{Est.on Q}
    \tb_{Q}[f]=\tb[f_Q]\le \Jb(Q)\|f\|^2_{H^{l}_{\mathrm{hom}}(Q)}, \quad f\in\Yc[\U],
 \end{equation}
holds, with some function of cubes $\Jb(Q)$. We need further  the function $\Jb$ to be \emph{upper semiadditive}; this means that if $Q_\io$ is a family of disjoint cubes, all of them inside a cube $Q^0$, then
\begin{equation}\label{semiadd}
    \sum \Jb(Q_\io)\le \Jb(Q^0).
\end{equation}
   We sum \eqref{Est.on Q} over all cubes in the covering $\U$ to obtain
 \begin{gather}\label{sumestim}
   \tb [f]\le\sum_{Q\in\U}\tb_{Q}[f]=\sum_{Q\in\U}\tb[f_{Q}]\le \sum_{Q\in\U}\Jb(Q)\|f\|^2_{H^{l}_{\mathrm{hom}}(Q)}\le\\\nonumber
   \le \max_{Q\in\U}\Jb(Q)\sum_{Q\in\U}\|f\|^2_{H^{l}_{\mathrm{hom}}(Q)}.
\end{gather}
 Now, suppose  that the covering $\U$ has a controlled finite multiplicity, i.e. every point in $\Mb$ is covered by no more than $\ka=\ka(\Nb)$ different cubes in $\U$.  Then the sum on right-hand side in \eqref{sumestim} can be majorated by $\ka \|f\|^2_{H^{l}_{\mathrm{hom}}(G)}$, which gives us
 \begin{equation}\label{aftersum}
     \tb_{Q}[f]\le \ka\max_{Q\in\U}\Jb(Q)\|f\|^2_{H^{l}_{\mathrm{hom}}(G)}.
 \end{equation}
 We finally arrive at the estimate
\begin{equation}\label{La Estim}
    \La{(\U)}\le \ka  \max_{Q\in\U}\Jb(Q).
\end{equation}
So, we reach our aim,  the estimate \eqref{BS.estim},  as soon as we construct the covering $\U$ such that
\begin{equation}\label{aim}
    (1)\ \ \ka  \max_{Q\in\U}\Jb(Q)<\la \quad  \mbox{ and } \quad (2)\ \ |\U|\db(\Nb,l)\le C\la^{-1}\|V\|^{(av,\Psi,\mu)}_{G}.
\end{equation}

 We set $\Jb(Q)=A_2\|V\|^{(av,\Psi,\mu)}_{ Q}$ with $A_2$ being the constant in \eqref{maz7}.  By Lemma \ref{Measlem4}, the required inequality \eqref{Est.on Q} is satisfied for this particular choice of $\Jb$. Next, this function $\Jb$
is upper semi-additive in the sense of \eqref{semiadd}.  This property of the averaged Orlicz norm  with respect to the Lebesgue measure was established in \cite[Lemma 3]{Sol94} and then extended to AR-measures,  see  Lemma 2.8 in \cite{KarSh}.

Now we construct the covering $\U(\la).$ First, by our Theorem \ref{mm},  there exists a cube $Q_0$ such that for every cube in $\R^\Nb$ with edges parallel to the ones of $Q_0$ (we call such cubes parallel to $Q_0$), its faces have zero
$\mu$-measure. We fix such a cube $Q_0$ and in the future \emph{all} cubes under our consideration will be parallel to $Q_0$. Consider a `large' cube $\Qb$ such that the cube concentric with $\Qb$ and with three times shorter edges still contains $G$ inside. We replace in  the  eigenvalue problem the domain $G$ by this larger cube $\Qb$. By the usual variational principle, any eigenvalue estimate in $\Qb$ implies automatically the same estimate for the initial problem in $G$.

For any point $X\in G$, we consider the family of cubes $Q_X(t)$, $0<t<\infty$, of size $t$ centered at $X$. By our choice of $Q_0$, the function $\m( Q_X(t))$ is continuous. Moreover, the function $\rho_X(t)=\Jb(Q_X(t))$ is a continuous function of $t$  and it tends to zero as $t\to 0$. The proof of this, rather elementary, fact is presented in \cite{KarSh}
for the case $\Nb=2, d=1$, and a very detailed proof for $d=\Nb$ is included in \cite{SZSol}. Both proofs carry over to our case automatically, without any modifications, since they are dimension-independent and it is only the continuity of $\m(Q_X(t))$ that is used there.

The function $\rho_{X}(t)$ stabilizes for  large $t$ to $\rho_{\infty}$, when  $Q_X(t)\supset \Qb$.
With some constant $\kb$, to be determined later, and $\la<\kb\rho_{\infty}$, we find  by continuity a value of $t=t(X)$ such that $\rho_X(t(X))=\la\kb^{-1}$.
The set of all  cubes $\{Q_{X}(t(X)), X\in G\}$ forms a covering of $\Mb$, and by the Besicovitch    covering lemma (see, e.g., \cite[Ch. 1, Theorem 1.1]{Guz})
one can find a finite sub-covering $\U=\U(\la)$ of finite multiplicity $\ka=\ka(\Nb)$; moreover, this covering can be split into finitely many  families, $\U_\nb, \nb\le\nb_0$, with $\nb_0$ depending only on the dimension $\Nb$, so that the cubes in each family are disjoint. This  will be the covering we are looking for.

In order to estimate the quantity of cubes in $\U$, we use the fact that the function $\Jb$ is upper semi-additive.  Therefore, for each of the families $\U_{\nb}$,
 \begin{equation*}
|\U_{\nb}|\kb^{-1}\la = \sum_{Q\in\U_{\nb}} \Jb(Q)\le \Jb(\Qb)=A_2\|V\|^{(av, \Psi,\mu)}_{G},\quad \nb=1,\dots,\nb_0,
 \end{equation*}
 and hence
 \begin{equation}\label{Theta}
    |\U|=\sum |\U_{\nb}|\le  \nb_0 \la^{-1} \kb A_2 \|V\|^{(av, \Psi,\mu)}_{G}.
 \end{equation}

Now we can choose the constant $\kb$ and arrive at the  required estimate.
Indeed, by \eqref{ineq.codim}, \eqref{La Estim},
\begin{equation}\label{subfinalLapl}
 n(\ka \la\kb^{-1},\Tb)\le \la^{-1}A_2 \nb_0 \kb \db(\Nb,l)\|V\|^{(av, \Psi, \mu)}_{G}.
\end{equation}
Taking $\kb=\ka$, we obtain estimate \eqref{BS.estim}.
\end{proof}
The approach, just presented, was called `piecewise polynomial approximation' by its authors. It measures how fast a function in the Sobolev space can be approximated in the weighted $L_2$ norm by functions that are polynomial on  cubes in the covering $\U.$
\subsection{Eigenvalue estimates}\label{EigEst}
Now we can conclude the \emph{ Proof}  of Theorem \ref{thm.eig.ahlf}.
Fix a bounded domain $G=\Om'\subset \Om $ with smooth boundary. We consider first the case when $\AF=\AF_0=(-\Delta_\Dc+1)^{-l/2}$, where  $(-\Delta_\Dc)$
 is the Laplace operator in $G$ with the Dirichlet boundary conditions. Having proved the eigenvalue estimate for this case,
  we will then use a simple argument to justify the required estimate for
a general operator $\AF$.

Consider the quadratic form $\Fb_V[u,\m, \AF_0]$, $u\in L_2(G)$. Denote $f=\AF_0 u$, so that $u=(-\Delta_\Dc+1)^{l/2}f,\ f\in\overset{\circ}H{}^{l}(G)$. The quadratic form  $\Fb_V[u,\m,\AF_0]$ is thus transformed to $\Fb_V^{(l)}[f]=\int_{\Mb}V|f|^2 d\m,\, f\in \overset{\circ}H{}^{l}(G)$. The operator $\overset{\circ}\Tb(V,\m)$ defined by the quadratic form $\Fb_V^{(l)}$ on
the Hilbert space $\overset{\circ}H{}^{l}(G)$ is exactly the operator considered in  Theorem \ref{BSESt}. Due to our substitution, $u=(-\Delta_\Dc+1)^{l/2}f$,
the operator $\overset{\circ}\Tb(V,\m)$ on $\overset{\circ}H{}^{l}(G)$ is similar to the  operator $\Tb(V,\Si, \AF_0)$ on $L_2(G)$, therefore they have the same spectrum, and the required estimate \eqref{main estmate} for this case follows. Note that the choice of $G=\Om'\subset\Om$ is arbitrary and may influence only the constant  in \eqref{main estmate}.


Now we pass to the general case, again for $V\ge0.$ As explained in Sections \ref{Setting} and \ref{Reductions}, we may assume that
the pseudodifferential operator $\AF$ contains cut-offs to $\Om'$.
 We denote by $\G$ the operator of restriction of functions in $H^l(\Om')$ to $\Mb$; this is a bounded operator
$\G: H^l(\Om)\to L_2(\Mb,\m)$ by Lemma \ref{meascor}. We set $|V|=W^2$, $W\ge0$. Then the quadratic form $\Fb_{|V|,\AF}[u]$ can be represented as follows
\begin{equation}\label{QuadrForm}
    \Fb_{V,\AF}[u] =\langle W\G\AF u, W\G\AF u\rangle_{L_2(\Mb,\m)}=\langle(W\G\AF)^{*}(W\G\AF) u,u\rangle_{L_2(\Om')}.
\end{equation}
Therefore, $\Tb(|V|,\m,\AF)=\AF^*\G^* |V|\G\AF$. Thus, compared with the operator $\Tb(|V|,\m,\AF_0)$, we have
\begin{equation}\label{GeneralOperator}
    \Tb(|V|,\m,\AF)=(\AF_0^{-1}\AF)^*\Tb(|V|,\m,\AF_0)(\AF_0^{-1}\AF).
\end{equation}
Since $\AF$ is a pseudodifferential operator of order $-l$, $\AF_0^{-1}\AF$ is a bounded operator on $L_2(\Om')$, together with  $(\AF_0^{-1}\AF)^*$.
Hence the eigenvalue estimate for $\Tb(|V|,\m,\AF_0)$, already justified, is preserved after multiplication by bounded operators.

\begin{proof}[Proof of Theorem \ref{ThEstimate}] It follows immediately from Theorem \ref{thm.eig.ahlf} since the measure $\mu_{\Si}$ on a compact Lipschitz  surface $\Si$ of dimension $d$ satisfies \eqref{Ahlfors} with $\a=d.$
\end{proof}

\section{Approximation of the weight} To perform the last reduction, we use systematically the asymptotic perturbation lemma  by M.Sh. Birman and M.Z. Solomyak (see, e.g., Lemma 1.5 in \cite{BS} or Lemma 6.1 in \cite{RT1}). By this fundamental lemma, if for an operator $\Tb$, there exist a family of approximating operators $\Tb_{\ve}, \, 0<\ve<\ve_0,$ such that for their eigenvalues  the asymptotic formula $n_{\pm}(\la,\Tb_{\ve} )\sim A_{\pm,\ve}\la^{-q}$ is known and for the difference $\Tb'_{\ve}=\Tb-\Tb_{\ve}$ the eigenvalue estimate $\limsup\la^{q}n(\la,\Tb'_{\ve})\le \ve$
 is proved, then the asymptotic formula $n_{\pm}(\la,\Tb)\sim A_{\pm}\la^{-q}$ holds with coefficients $A_{\pm}=\lim A_{\pm,\ve}$. In our case, the estimates will be provided by Theorem \ref{ThEstimate}. So, let $V\in L^{\Psi}(\Si)$. We approximate $V$ by sufficiently regular functions.
\begin{lem}\label{Lem.App.1} For any $\ve>0$, there exists $V_{\ve}\in C_0^{\infty}(\Om)$ such that $\|V-V_\ve\|^{(av,\Psi)}_{\Si}<\ve$.
\end{lem}
\begin{proof}  For $d=\Nb$ this is a well known statement about the density of smooth functions. So, let $d<\Nb$.  As usual, we can assume that the surface $\Si$ is covered by
one local chart $\Si: \yb=\vf(\xb),\ \xb\in\Dc\subset \R^{d}$. First, we truncate the function $V$ at some level, i.e. we set
$$V'_\ve(X)=V(X) \ \mbox{ if } \  |V(X)|\le N,\quad V'_\ve(X)=N\sign{V(X)}  \ \mbox{ otherwise} .$$
 The function $V_\ve'$ is bounded and, since $V\in L^{\Psi}(\Si),$ the cut-off level $N$ can be chosen so that $\|V-V_\ve'\|^{(av,\Psi)}_\Si <\ve/3$. Next, we extend the function $V_\ve'$ defined on $\Si$ to $\Dc\times\R^\dF$ by setting $V_\ve''(\xb,\yb)=V_\ve'(\xb,\vf(\xb))$, $\xb\in\Dc$. On $\Si$, we still  have $\|V-V_\ve''\|^{av,\Psi}_\Si <\ve/3.$ The resulting function $V_\ve''$ belongs to $L_2(\Dc\times (-r,r))$ for any $r>0$. Now consider the convolution  of $V_{\ve}''$ with some mollifier $\om(\xb)$. We obtain a smooth function $V_{\ve}'''$ depending, again, on $\xb$ only. The mollifier $\om$ can be chosen so that $\|V_\ve''-V_\ve'''\|^{(av,\Psi)}_{\Si}\le C\|V_\ve''(\cdot, \yb)-V_\ve'''(\cdot, \yb)\|_{L_2(\Dc)}<\ve/3$. Finally, multiplying
 $V_\ve'''$ by a smooth compactly supported cut-off function $\chi$, which equals $1$ near $\Si$, we obtain the required approximation $V_{\ve}$.
\end{proof}

We need approximating functions to satisfy an additional property (cf. Corollary \ref{cor.plusminus}).
\begin{lem}\label{Lem.Appr.2}  The function $V_{\ve}$ can be approximated by a function $\tilde{V}_\ve\in C_0^{\infty}(\Om)$ such that $\|V_{\ve}-\tilde{V}_{\ve}\|^{av,\Psi}_{\Si}<\ve$ and there exist open sets $\Om_{\pm}\subset \Om$ such that $\pm \tilde{V}_\ve(X)\ge 0$ for $X\in \Om_{\pm}$, $\tilde{V}_\ve(X)=0$ for $X\in\Om\setminus (\Om_+\cup\Om_-)$, and, finally,
$\dist (\Om_+,\Om_-)>0$.
\end{lem}
\begin{proof} Let the cylinder $\Dc\times [-N,N]$ contain the support of $V_{\ve}$. Consider the  closed set $\Om_0=\{X\in\overline{\Dc}\times [-N,N]:V_{\ve}(X)=0 \}$.
Since $V_{\ve}$ is uniformly continuous on the compact set $\overline{\Dc}\times [-N,N]$,  for any $\de_1>0$ there exists $\de_2>0$ such that $|V_{\ve}(X)|<\de_1$ for $X$ in the $\de_2$ - neighborhood of $\Om_0$. Take a  function $\chi_{\ve}\in C_0^\infty(\Om) $ such that $\chi_{\ve}(X)\in[0,1]$, $\chi_{\ve}(X)=0$ in the $\de_2/2$ neighborhood of $\Om_0$ and $\chi_{\ve}(X)=1$ outside the $\de_2$ neighborhood of $\Om_0$. Set $\tilde{V}_{\ve}=\chi_{\ve}{V}_{\ve}$. Then the sets $\Om_{\pm}=\{X: \pm \tilde{V}_{\ve}(X)>0\}$ satisfy the conditions of Lemma. Indeed, if $|X_+-X_-|<\de_2/2$ and $\pm \tilde{V}_{\ve}(X_\pm)>0$, then, by continuity of  $V_{\ve}$, there must exist a point $X_0\in\Om_0$ on the straight interval connecting $X_+$ and $X_-$, such that at least one of the distances $|X_\pm-X_0|$ is smaller than $\de_2/2$, and this contradicts the construction of $\Om_{\pm}$. Moreover, $|V_{\ve}-\tilde{V}_{\ve}|\le \de_1$ everywhere in $\Om$, and, by choosing $\de_1$ sufficiently small, we obtain the required approximation property.
\end{proof}

\section{The Asymptotic Formula}\label{7}
We now present the proof of Theorem \ref{ThmAs} for the case when there is just one surface $\Si$.
 \begin{proof} By the asymptotic perturbation lemma, together with our Lemmas \ref{Lem.App.1}, \ref{Lem.Appr.2} and Theorem \ref{ThEstimate}, it is sufficient to prove the eigenvalue asymptotics formula for the function $\tilde{V}_{\ve},$ $\ve>0,$ as above. We will denote it simply by $V$ in what follows.

By Corollary \ref{cor.plusminus}, to study the asymptotics of positive (negative) eigenvalues of $\Tb(V)$, it is sufficient to consider the operator with a non-negative (non-positive) function $V$.

So, with a smooth $V\ge0,$ we perform a reduction to the integral operator. We have already made the first step. We set $V=W^2 $ with $W\in C_0^\infty$, $W\ge0$, and  denote by $\G$ the operator of restriction of functions in  $H^{l}(\Om)$  to the surface $\Si$, $\G:H^{l}(\Om)\to L_2(\Si,\mu_\Si).$ The form \eqref{eq1} can now be written as

$    \Fb_V[u]=\langle(W\G\AF)u, (W\G\AF)u\rangle_{L_2(\Si,\mu_\Si)} =\langle(W\G\AF)^*(W\G\AF)u,u\rangle_{L_2(\Om)},$
with $(W\G\AF)$ considered as acting from $L_2(\Om)$ to $L_2(\Si,\mu_\Si)$.

It follows that $\Tb(V)=(W\G\AF)^*(W\G\AF)$.  The nonzero eigenvalues of the operator $\Tb$ in ${L_2(\Om)}$ coincide with the nonzero eigenvalues of the operator  $\Lb=(W\G\AF)(W\G\AF)^*=W[\G(\AF\AF^*) \G^*] W$ in $L_2(\Si,\mu_\Si)$. Here $\AF\AF^*$ is, up to a smoothing term, a nonnegative pseudodifferential operator of order $-2l=-\Nb$, with the principal symbol $|a_{-l}(X,\Xi)|^2$, thus it is a weakly polar integral operator. The corresponding Schwartz kernel  $\Lc(x,y)$ has the leading singularity $A(X)\log|x-y|+\vartheta(X,X-Y)$,
where $\vartheta(X,X-Y)$ is a function positively homogeneous in $X-Y$ of order zero and smooth in the first variable. Therefore, the operator $\Lb$ is an integral operator on $\Si$,
 \begin{equation}\label{eq3}
 (\Lb v)(X)=\int_\Si W(X)\Lc(X,Y)W(Y)v(Y) d\mu_{\Si}(Y),
\end{equation}
see, e.g., \cite{Taylor}, Ch. 2, especially, Proposition 2.6. The function $W$, defined initially on $\Si$, is, in fact, the trace on $\Si$ of a smooth function $\tilde{W}(X)=(\tilde{V}(X))^{1/2}$ defined on $\R^{\Nb}$. Hence, the operator in \eqref{eq3} takes the form
$  (\Lb v)(X)=\int_\Si \tilde{\Kc}(X,Y)v(Y) d\mu_{\Si}(Y)  ,$
where $\tilde{\Kc}$ is the integral kernel of the pseudodifferential operator  with the principal symbol $\tilde{V}(X)|a_{-l}(X,\Xi)|^2$.
We are now in the setting of the paper \cite{RT} (see also \cite{RT1}). By Theorem 6.4 there, the asymptotics of the eigenvalues of a weakly polar integral operator on a Lipschitz surface is given exactly by formula \eqref{eq5}. We mention  that the proof in those papers uses the asymptotic perturbation   Lemma 1.5 in \cite{BS} in the analysis of operator convergence when a Lipschitz surface is approximated by smooth ones in a special way. Thus, this beautiful invention by M. Birman and M. Solomyak is used twice in our study, in quite different settings.
\end{proof}
 Now we  justify the eigenvalue asymptotic formula \eqref{As.several} for the case when there are several Lipschitz surfaces of possibly different dimensions,  including, possibly,
dimension $d = \Nb$, codimension $\dF=0$, i.e. a domain in $\Om$ with an absolutely continuous measure. In order to avoid excessive complications in the proof, we impose a geometric restriction.

\begin{thm}\label{Thm.as.several}
 Let $\Si_j$, $j=1,\dots,J,$ be compact Lipschitz surfaces of dimension $d_j, 1\le d_j\le \Nb,$ with measures $P_j=V_j\m_{\Si_j}$, $V_j\in L^{\Psi}(\m_{\Si_j})$. We assume that the surfaces of the same dimension are disjoint. Then  the eigenvalues of  the operator $\Tb(P,\AF)=\sum\Tb(P_j,\AF)$ satisfy the asymptotic formula
 \begin{equation}\label{AsFormula}
 n_{\pm}(\la,\Tb(P,\AF))\sim \la^{-1}\sum_j C^{\pm}(V_j,\Si_j,\AF),
  \end{equation}
  where $C^{\pm}(V_j,\Si_j,\AF)$ are given by \eqref{eq5} with $V_j, \Si_j$ in place of $V, \Si$.
  \end{thm}

 \begin{proof}
 If the surfaces $\Si_j$ are disjoint, so that all mutual distances are positive, the result follows by the inductive application of Lemma \ref{LemLocalization}. Otherwise, we construct an approximation of the measure $P=\sum P_j$ by measures with disjoint surfaces $\Si_j$.

 Relabeling the surfaces if necessary, we can assume that their dimensions increase with $j$, i.e. $d_1 \le d_2 \le \cdots \le d_J$.
 We start with  $\Si_1$. Consider the $\de$-neighborhood $U_1(\de)$ of $\Si_1$ in $\R^{\Nb}$.
 If $d_j > d_1$, then the surface measure $\m_{\Si_j}(\Si_j\cap U_1(\de))$, $j>1$, decays at least as $\de^{d_j-d_1}$ for $\de\to 0$. If $d_j = d_1$, then, according to our assumption,
 $\Si_j\cap U_1(\de) = \emptyset$ for all sufficiently small $\de > 0$. Either way,
 $\m_{\Si_j}(\Si_j\cap U_1(\de))$, $j>1$, tends to zero as $\de\to 0$. Therefore, for $\de$ small enough, the averaged $\Psi$-norm of each $V_j$, $j>1,$ over  $\Si_j\cap U_1(\de)$ can be made arbitrarily small. We set $V_j^{(1)}(X)=V_j(X) (1-\chi(X))$, $j>1$, where $\chi$ is the characteristic function of $U_1(\de)$. Then the averaged norm of $V_j-V_j^{(1)}$ is small and $\Si_1$ is separated from the support of $V_j^{(1)},\, j>1.$ Next we consider $\Si_2^{(1)}=\Si_2\cap\supp V_2^{(1)}$. This piece of the surface $\Si_2$ is separated from $\Si_1$, but may cross $\Si_j$, $j>2$. We repeat with $\Si_2^{(1)}$ the same procedure as above, considering its sufficiently small neighborhood in $\R^\Nb$ and then killing the remaining measures in this neighborhood. In this way, after a finite number of steps, we arrive at a system of separated measures, to which Lemma \ref{LemLocalization} can be applied, with further application of the already proved case of Theorem \ref{ThmAs}. In this construction, we  introduce  perturbations of $V_j$ with small averaged  $\Psi$-norms. By Theorem \ref{ThEstimate} and, again, Lemma 1.5 in \cite{BS}, the asymptotic eigenvalue formula is thus justified in full.
 \end{proof}



\begin{acknowledgments}
G.R. : The paper received a support from the RScF (Project 20-11-20032)
\end{acknowledgments}

\end{document}